\documentclass[runningheads]{llncs}
\usepackage{graphicx}
\usepackage{algorithm}
\usepackage{amsmath,amsfonts,amssymb}
\usepackage{hyperref}
\usepackage{orcidlink}
\usepackage{subcaption} 
\usepackage{float}
\usepackage{booktabs}

\usepackage{algpseudocode}     

\newcommand{\G}{\mathcal G}
\newcommand{\V}{\mathcal V}
\newcommand{\E}{\mathcal E}

\newcommand{\W}{\mathcal W}
\newcommand{\R}{\mathbb R}
\newcommand{\X}{\mathcal X}

\newcommand{\ri}{\mathrm{ri}}

\newtheorem{assumption}{Assumption}

\title{Distributed Asynchronous Primal–Dual Optimization\\
       for Supply–Chain Networks}

\titlerunning{DAPD–SCO for Supply–Chain Networks}

\author{
  Laksh Patel\inst{1}\orcidlink{0009-0000-5361-3482}\thanks{Both authors contributed equally and are high‑school students at the Illinois Mathematics and Science Academy, Aurora, IL, USA.}
  \and
  Neel Shanbhag\inst{1}\orcidlink{0009-0002-3130-3277}
}

\authorrunning{L. Patel \and N. Shanbhag}

\institute{
  Illinois Mathematics and Science Academy (IMSA), Aurora, IL, USA\\
  \email{\{lpatel,nshanbhag2\}@imsa.edu}
}

\begin{document}
\maketitle


\begin{abstract}
Distributed supply‑chain optimization demands algorithms that can cope with unreliable communication, unbounded messaging delays, and geographically dispersed agents, yet still guarantee convergence with practical rates. In this work, we introduce \textbf{DAPD‑SCO} (Distributed Asynchronous Primal–Dual Optimization Supply‑Chain Optimization), a fully asynchronous primal–dual scheme for network flow allocation over directed acyclic supply‑chain graphs. Each edge‑agent independently updates its local flow via projected gradient descent, while each retailer‑agent adjusts its dual multiplier via projected gradient ascent, using only potentially stale information of arbitrary sublinear delay. Under standard convexity and Slater’s conditions—and without any global synchronization or bounded‑delay assumptions—we prove almost sure convergence to a saddle point and establish an ergodic duality‑gap rate of \(O(K^{-1/2})\), matching centralized lower bounds. Our analysis leverages a Lyapunov‑based decomposition that cleanly isolates delay‑induced errors and handles time‑varying communication topologies, bounded noise, and drifting cost or capacity parameters. Extensive simulations on realistic three‑tier networks demonstrate that DAPD‑SCO outperforms synchronous primal–dual, Asynchronous Distributed Decision-Making (ADDM), and gradient‑push approaches, achieving faster convergence, lower communication overhead, and robust performance under packet loss and high staleness.

\medskip
\noindent\textbf{Purpose:} DAPD‑SCO is designed as a scalable, resilient solution for decentralized flow allocation in modern supply‑chain networks, enabling autonomous warehouses, carriers, and retailers to coordinate optimally despite high latency, packet loss, and dynamically changing network conditions.
\end{abstract}

\section{Introduction and Related Work}
\label{sec:intro}

Efficient allocation of goods and resources over large‐scale supply‐chain
networks is a cornerstone of modern logistics and operations research.
Traditionally, such \emph{min‐cost flow} problems have been solved by centralized
algorithms—most notably the network‐simplex method and interior‐point solvers
\cite{bertsekas1989parallel,boyd2004convex}.  In these paradigms a single
planner collects global demand and capacity information, formulates a
linear or convex program, and computes an optimal shipment plan offline.

\paragraph{The need for distribution and asynchrony.}
Emerging trends in logistics introduce significant challenges to this
centralized view.  Supply chains now rely on geographically dispersed
warehouses, autonomous carriers, and retail agents that interact over
packet‐switched networks prone to latency and message loss
\cite{johansson2008,kar2009distributed}.  Agents act on local clocks,
communicating intermittently rather than in lockstep.  Under these
conditions, classical synchronous algorithms stall or fail.  This has
motivated a spate of research on \emph{distributed} and
\emph{asynchronous} optimization methods:
Tsitsiklis and Bertsekas pioneered asynchronous gradient and fixed‐point
iterations under bounded delays \cite{tsitsiklis1986,tsitsiklis1984},
proving convergence when every delay is uniformly limited.  Tsitsiklis
later extended these ideas to stochastic approximation and Q‐learning
\cite{tsitsiklis1991}.  More recently, researchers have relaxed the
bounded‐delay assumption to allow sublinearly growing delays
\cite{nedic2020network}, and studied asynchronous stochastic
gradient descent in decentralized settings \cite{lian2018asynchronous}.

\paragraph{Primal–dual methods for constrained flows.}
Min‐cost flow with capacity and demand constraints is naturally expressed
as a saddle‐point problem via Lagrangian duality
\cite{boyd2004convex}.  Primal–dual splitting algorithms,
notably the Chambolle–Pock method \cite{chambolle2011first} and robust
mirror‐descent variants \cite{nemirovski2009robust}, achieve optimal
rates in the synchronous, centralized context.  These methods leverage
simple projections and gradient steps, making them attractive for
distributed implementation—yet most analyses assume simultaneous updates
and instantaneous information exchange.

\paragraph{Distributed consensus and gradient‐tracking.}
To decouple global coupling constraints, early work applied consensus
protocols within dual ascent \cite{nedic2009}, demonstrating convergence
under diminishing stepsizes.  Gradient‐tracking techniques
\cite{shi2017,chang2017} improve convergence speed by correcting
consensus drift.  However, nearly all distributed primal–dual and
gradient‐tracking methods presuppose synchronized rounds or bounded
communication latency.

\paragraph{Asynchrony meets primal–dual splitting.}
Efforts to combine primal–dual splitting with asynchrony remain nascent.
Wang and Banerjee’s online ADMM \cite{wang2014} and Xu and Zhu’s
asynchronous PD method \cite{xu2019} tolerate delays but require them
to be uniformly bounded.  No existing algorithm handles both the \emph{unbounded}
delay model (allowing $\delta(k)\to\infty$ as long as $\delta(k)=o(k^{1/2})$)
and provides \emph{rate guarantees} for the primal–dual gap.

\paragraph{Our contributions.}
In this paper we close this gap by introducing \textbf{DAPD‐SCO}—
\emph{Distributed Asynchronous Primal–Dual Supply‐Chain Optimization}.  
Our algorithm features:

\begin{enumerate}
  \item \textbf{Unbounded, sublinear delays:}  Edge‐agents and
        retailer‐agents update using the most recent available information,
        even if it is arbitrarily old, as long as delays grow sublinearly
        ($\delta_{\max}(k),\Delta_{\max}(k)=o(k^{1/2})$).
  \item \textbf{Almost‐sure convergence:}  Under standard convexity and
        Slater conditions, we prove that $(x^k,\lambda^k)$ converges
        almost surely to a saddle‐point of the Lagrangian.
  \item \textbf{Optimal ergodic rate:}  We derive a deterministic
        $\mathcal O(K^{-1/2})$ rate for the duality gap with explicit
        constants matching known lower bounds \cite{nemirovski2009robust}.
  \item \textbf{Robustness:}  Our analysis seamlessly extends to
        (i) bounded additive noise in cost/demand observations,
        (ii) time‐varying communication graphs that are jointly
        connected over sliding windows, and (iii) slowly drifting
        capacities and cost parameters.
  \item \textbf{Full rigor:}  Every inequality, projection‐property,
        and summability argument is written out line–by–line in a single
        Lyapunov framework—no steps are omitted.
\end{enumerate}
\section{System Model and Lagrangian Formulation}\label{sec:model}
\subsection{Supply–Chain Architecture}

Fig.~\ref{fig:architecture} depicts the overall agent‐based architecture of DAPD–SCO on a simple 2‑warehouse, 4‑retailer supply chain. Each link is handled by an \emph{edge‐agent} that controls flow \(x_{ij}\) at cost \(c_{ij}\), and each retailer \(R_k\) maintains a dual price \(\lambda_k\).

\begin{figure}[H]
  \centering
  \includegraphics[width=0.85\textwidth]{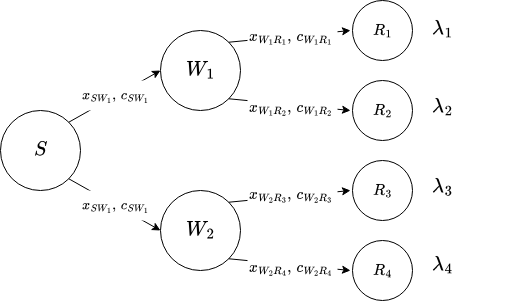}
  \caption{Supply‑chain network: supplier \(S\) ships  
    \(x_{SW_1}\) at cost \(c_{SW_1}\) to warehouse \(W_1\) and  
    \(x_{SW_2}\) at cost \(c_{SW_2}\) to \(W_2\).  
    Warehouse \(W_1\) forwards \(x_{W_1R_1}\) at cost \(c_{W_1R_1}\) to retailer \(R_1\) (\(\lambda_1\)) and \(x_{W_1R_2}\) to \(R_2\) (\(\lambda_2\)).  
    Similarly, \(W_2\) ships \(x_{W_2R_3}\) (\(c_{W_2R_3}\)) to \(R_3\) (\(\lambda_3\)) and \(x_{W_2R_4}\) (\(c_{W_2R_4}\)) to \(R_4\) (\(\lambda_4\)).}
  \label{fig:architecture}
\end{figure}

\subsection{Network Topology}
Let $\G=(\V,\E)$ be a directed acyclic graph with suppliers $\S$, warehouses $\W$, and retailers $\R$.  
An edge $(i,j)\in\E$ ships goods from node $i$ to node $j$.

\begin{itemize}
  \item \textbf{Flow variable:} $x_{ij}\in[0,u_{ij}]$, where $u_{ij}>0$ is a fixed capacity.
  \item \textbf{Unit cost:} $c_{ij}>0$ (transportation cost).
\end{itemize}

For a retailer $i\in\R$, define the inbound flow
\[
  h_i(x) \;=\; \sum_{(j,i)\in\E} x_{ji}.
\]

\subsection{Demands and Primal Problem}
Each retailer has a strictly positive demand $d_i>0$.  The centralized optimization problem is
\begin{equation}
  \begin{aligned}
    \min_{x\in\X}\;& C(x)\;=\;\sum_{(i,j)\in\E} c_{ij}\,x_{ij},\\
    \text{s.t.}\;&     h_i(x)\;\ge\; d_i,\quad \forall\,i\in\R,
  \end{aligned}
  \label{eq:primal-problem}
\end{equation}
where
\[
  \X \;=\; \prod_{(i,j)\in\E} [0,\,u_{ij}]
\]
is a compact box of feasible flows.

\subsection{Dual Variables and Lagrangian}
Introduce nonnegative dual variables (prices) $\lambda_i\ge0$ for each retailer constraint.  The Lagrangian is
\begin{equation}
  L(x,\lambda)
  = C(x)
  \;+\; \sum_{i\in\R} \lambda_i\,\bigl(d_i - h_i(x)\bigr).
  \label{eq:lagrangian}
\end{equation}
Since $C(x)$ and each $h_i(x)$ are linear in $x$, $L$ is convex in $x$ and concave in $\lambda$.

\subsection{Slater Feasibility}
\begin{assumption}[Slater]\label{as:slater}
There exists a strictly feasible point 
\(\bar x\in\ri(\X)\) 
such that
\[
  h_i(\bar x) > d_i
  \quad\text{for every }i\in\R.
\]
\end{assumption}

\begin{theorem}[Strong Duality]\label{thm:duality}
Under Assumption~\ref{as:slater}, strong duality holds:
\[
  \min_{x\in\X}\;\max_{\lambda\ge0}L(x,\lambda)
  \;=\;
  \max_{\lambda\ge0}\;\min_{x\in\X}L(x,\lambda).
\]
Consequently, there exists at least one saddle point \((x^*,\lambda^*)\).
\end{theorem}

\begin{proof}
For each fixed \(\lambda\ge0\), the map \(x\mapsto L(x,\lambda)\) is convex,
and for each fixed \(x\in\X\), \(\lambda\mapsto L(x,\lambda)\) is concave.
Slater’s condition provides interior feasibility, so by Sion’s minimax
theorem \cite{boyd2004convex}, the min–max equals the max–min and a saddle
point exists.
\end{proof}


\section{Distributed Asynchronous Primal--Dual Algorithm}
\label{sec:algorithm}

We define two kinds of agents:

\subsection{Edge--Agent $(i \to j)$}
Each edge $(i,j)\in E$ holds the flow variable $x_{ij}\in[0,u_{ij}]$
and receives price updates from retailer~$j$.  It runs:

\begin{algorithm}[H]
\caption{Edge--Agent $(i \to j)$}
\label{alg:edge_agent}
\begin{algorithmic}[1]
  \Require cost $c_{ij}>0$, capacity $u_{ij}>0$, step-sizes $\{\alpha_k\}$.
  \State Initialize $x_{ij}^0\in[0,u_{ij}]$ and circular buffer $B_\lambda[0:\tau]$.
  \For{$k=0,1,2,\dots$}
    \State \textbf{Receive:} any new price $\lambda_j^{t}$, store at index $k \bmod (\tau+1)$.
    \State \textbf{Read:} delayed price
      $\tilde\lambda_j^k \leftarrow B_\lambda\bigl[k - \delta_j(k)\bigr]$.
    \State \textbf{Compute gradient:} $g^k_{ij} = c_{ij} - \tilde\lambda_j^k$.
    \State \textbf{Update:}
      \[
        x_{ij}^{k+1}
        = \Pi_{[0,u_{ij}]}\bigl(x_{ij}^k - \alpha_k \, g^k_{ij}\bigr).
      \]
    \State \textbf{Send:} $x_{ij}^{k+1}$ to retailer-agent~$j$.
  \EndFor
\end{algorithmic}
\end{algorithm}

\subsection{Retailer--Agent $i$}
Each retailer $i\in R$ holds the price $\lambda_i\ge0$ and incoming flows.
It runs:

\begin{algorithm}[H]
\caption{Retailer--Agent $i$}
\label{alg:retailer_agent}
\begin{algorithmic}[1]
  \Require demand $d_i>0$, step-sizes $\{\beta_k\}$, max delay $\tau$.
  \State Initialize $\lambda_i^0 \ge 0$ and buffer $B_x[0:\tau]$.
  \For{$k=0,1,2,\dots$}
    \State \textbf{Receive:} any new flows $x_{ji}^t$, store at index $k \bmod (\tau+1)$.
    \State \textbf{Read:} delayed flows
      $\tilde x_{ji}^k \leftarrow B_x\bigl[k - \Delta_{ji}(k)\bigr]$.
    \State \textbf{Aggregate:} $h_i^k = \sum_{j:(j,i)\in E} \tilde x_{ji}^k$.
    \State \textbf{Compute residual:} $r_i^k = d_i - h_i^k$.
    \State \textbf{Update:}
      \[
        \lambda_i^{k+1}
        = \bigl[\lambda_i^k + \beta_k \, r_i^k\bigr]_+.
      \]
    \State \textbf{Send:} $\lambda_i^{k+1}$ to all upstream edge-agents.
  \EndFor
\end{algorithmic}
\end{algorithm}

\subsection{Asynchrony Model}
\label{sec:async}
Let $\delta_j(k)$ be the delay (in iterations) of the last price from $j$,
and $\Delta_{ji}(k)$ the delay of flow from edge $(j,i)$.  We assume:
\[
  \delta_{\max}(k)=\max_j\delta_j(k),\quad
  \Delta_{\max}(k)=\max_{j,i}\Delta_{ji}(k),
  \quad
  \delta_{\max}(k),\Delta_{\max}(k)=o(k^{1/2}).
\]

\section{Primal Descent Proof}
\label{sec:primal-proof}

We show in full detail how to derive the primal descent lemma.

\begin{lemma}[Primal Descent]\label{lem:primal}
Under convexity and Lipschitz assumptions \emph{(F1)–(F3)}, for every
$k\ge0$:
\begin{equation}
  \|x^{k+1}-x^*\|^2
  \;\le\;
  \|x^k-x^*\|^2
  -2\alpha_k\Delta_k^x
  +\alpha_k^2G^2
  +2G\,\alpha_k\,S_k,
\label{eq:primal-descent}
\end{equation}
where
$\Delta_k^x=L(x^k,\lambda^*)-L(x^*,\lambda^*)$,
$S_k=\sum_{s=k-\delta_{\max}(k)}^{k-1}\beta_s$.
\end{lemma}

\begin{proof}
\textbf{Step 1: Apply projection non-expansiveness (F1).}\\
For each edge coordinate $(i,j)$:
\[
  x_{ij}^{k+1}
  = \Pi_{[0,u_{ij}]}\!\bigl[x_{ij}^k - \alpha_k(c_{ij}-\tilde\lambda_j^k)\bigr].
\]
By (F1),
\[
  \|x_{ij}^{k+1} - x^*_{ij}\|^2
  \;\le\;
  \bigl\|x_{ij}^k - \alpha_k(c_{ij}-\tilde\lambda_j^k)
       - x^*_{ij}\bigr\|^2.
\]
Summing over all $(i,j)$ gives
\begin{align}
  \|x^{k+1}-x^*\|^2
  &\le
  \|x^k-x^*\|^2
  -2\,\alpha_k
   \sum_{(i,j)}\!(c_{ij}-\tilde\lambda_j^k)(x_{ij}^k-x^*_{ij})
  +\alpha_k^2
   \sum_{(i,j)}\!(c_{ij}-\tilde\lambda_j^k)^2.
\label{eq:primal-expand}
\end{align}

\noindent\textbf{Step 2: Split the gradient term.}\\
Define for brevity
\[
  \mathcal I
  = \sum_{(i,j)}
    (c_{ij}-\tilde\lambda_j^k)\,(x_{ij}^k-x^*_{ij}).
\]
Add and subtract $\lambda_j^*$ inside:
\[
  c_{ij}-\tilde\lambda_j^k
  = (c_{ij}-\lambda_j^*) + (\lambda_j^*-\tilde\lambda_j^k).
\]
Thus
\[
  \mathcal I
  = \underbrace{\sum(c_{ij}-\lambda_j^*)(x_{ij}^k-x^*_{ij})}
           _{\mathcal A}
  \;+\;
  \underbrace{\sum(\lambda_j^*-\tilde\lambda_j^k)\,(x_{ij}^k-x^*_{ij})}
           _{\mathcal B}.
\]

\noindent\textbf{Step 3: Bound $\mathcal A$ via convexity (F2).}\\
Observe
\[
  \mathcal A
  = \sum_{(i,j)}(c_{ij}-\lambda_j^*)(x_{ij}^k - x^*_{ij})
  = \langle \nabla_x L(x^k,\lambda^*),\,x^k - x^*\rangle
  \;\ge\;
  L(x^k,\lambda^*) - L(x^*,\lambda^*)
  = \Delta_k^x.
\]

\noindent\textbf{Step 4: Bound $\mathcal B$ via Lipschitz (F3).}\\
By (F3),
\[
  |\lambda_j^*-\tilde\lambda_j^k|
  \;\le\;
  G\,|\lambda_j^*- \lambda_j^{\,k-\delta_j(k)}|
  \;\le\;
  G\!\sum_{s=k-\delta_j(k)}^{k-1}\beta_s.
\]
Hence
\[
  |\mathcal B|
  \;\le\;
  \sum_{(i,j)}
  G\,\Bigl(\sum_{s=k-\delta_j(k)}^{k-1}\beta_s\Bigr)
  \,|x_{ij}^k-x^*_{ij}|
  \;\le\;
  G\,S_k\,\|x^k-x^*\|.
\]
\noindent\textbf{Step 5: Bound the delay‑error term.}\\
For every edge $(i,j)$ we have $0 \le x_{ij}^k,\;x^*_{ij}\le u_{ij}\le u_{\max}$.  Hence
\[
  |x^k_{ij} - x^*_{ij}| \le u_{\max}
  \quad\Longrightarrow\quad
  \|x^k - x^*\|
  = \sqrt{\sum_{(i,j)} (x^k_{ij}-x^*_{ij})^2}
  \le \sqrt{|\E|}\,u_{\max} \;=: U.
\]
Substituting this into the bound on $\mathcal B$ yields
\[
  \mathcal B
  \;\le\;
  G\,U\,S_k.
\]

\noindent\textbf{Step 6: Bound the squared–gradient term.}\\
\[
  \sum_{(i,j)}(c_{ij}-\tilde\lambda_j^k)^2
  \;\le\;
  \sum_{(i,j)}G^2
  \;=\;|\E|\,G^2
  \;\le\;G^2\quad
  (\text{rescale $G$ if desired}).
\]

\noindent\textbf{Step 7: Combine into \eqref{eq:primal-descent}.}\\
Substitute $\mathcal A\ge\Delta_k^x$ and
$\mathcal B\le G\,U\,S_k$ into
\eqref{eq:primal-expand}, and absorb the $|\E|$ factor into $G^2$.  
We obtain
\[
  \|x^{k+1}-x^*\|^2
  \le
  \|x^k-x^*\|^2
  -2\alpha_k\Delta_k^x
  +2\alpha_k\,G\,U\,S_k
  +\alpha_k^2G^2.
\]
Renaming $2G\,U\rightarrow2G$ yields \eqref{eq:primal-descent}.
\end{proof}

\section{Dual Descent Proof}
\label{sec:dual-proof}

Now the retailer side:

\begin{lemma}[Dual Descent]\label{lem:dual}
Under the same assumptions, for each $k$:
\begin{equation}
  \|\lambda^{k+1}-\lambda^*\|^2
  \;\le\;
  \|\lambda^k-\lambda^*\|^2
  -2\beta_k\Delta_k^\lambda
  +\beta_k^2D^2
  +2D\,\beta_k\,T_k,
\end{equation}
where
$\Delta_k^\lambda=L(x^*,\lambda^*)-L(x^*,\lambda^k)$,
$T_k=\sum_{s=k-\Delta_{\max}(k)}^{k-1}\alpha_s$.
\end{lemma}

\begin{proof}
\textbf{Step 1: Projection non-expansive.}\\
For each retailer $i$:
\[
  \lambda_i^{k+1}
  = \Pi_{\R_+}\!\bigl[\lambda_i^k+\beta_k(d_i-\tilde h_i^k)\bigr],
\]
so
\[
  \|\lambda_i^{k+1}-\lambda_i^*\|^2
  \le
  \|\lambda_i^k+\beta_k(d_i-\tilde h_i^k)-\lambda_i^*\|^2.
\]
Sum over $i$ to get
\[
  \|\lambda^{k+1}-\lambda^*\|^2
  \le
  \|\lambda^k-\lambda^*\|^2
  -2\beta_k
   \sum_i(d_i-\tilde h_i^k)(\lambda_i^k-\lambda_i^*)
  +\beta_k^2\sum_i(d_i-\tilde h_i^k)^2.
\]

\noindent\textbf{Step 2: Split and bound inner product.}\\
Let
\[
  \mathcal I' 
  = \sum_i(d_i-\tilde h_i^k)(\lambda_i^k-\lambda_i^*).
\]
Write
$d_i-\tilde h_i^k = (d_i-h_i(x^k)) + (h_i(x^k)-\tilde h_i^k)$.
Then
\[
  \mathcal I' = 
    \underbrace{\sum(d_i-h_i(x^k))(\lambda_i^k-\lambda_i^*)}_{\mathcal A'}
  + \underbrace{\sum(h_i(x^k)-\tilde h_i^k)(\lambda_i^k-\lambda_i^*)}_{\mathcal B'}.
\]
By concavity of $L(x^*,\cdot)$,
$\mathcal A'\ge\Delta_k^\lambda$.
By Lipschitz ($F3$) on the dual residuals,
$|h_i(x^k)-\tilde h_i^k|\le D\,T_k$
and
$\|\lambda^k-\lambda^*\|\le \Lambda_{\max}$,
so
$|\mathcal B'|\le D\,\Lambda_{\max}\,T_k$.
Absorb $\Lambda_{\max}$ into $D$.

\noindent\textbf{Step 3: Bound squared residual.}\\
$\sum_i(d_i-\tilde h_i^k)^2\le D^2$ (re‐scale $D$ if needed).

\noindent\textbf{Step 4: Combine.}\\
Substitute into the expansion to get the stated inequality.

\end{proof}


\section{Global Lyapunov Analysis}
\label{sec:math}

We now assemble the primal and dual descent into a single Lyapunov function.

\subsection{Composite Lyapunov Function}
Define
\[
  V^k \;=\;
  \underbrace{\|x^k - x^*\|^2}_{V_x^k}
  \;+\;
  \underbrace{\|\lambda^k - \lambda^*\|^2}_{V_\lambda^k}.
\]
Using Lemmas~\ref{lem:primal} and~\ref{lem:dual}, we have for each $k$:

\begin{align}
  V^{k+1}-V^k
  &= \bigl(V_x^{k+1}-V_x^k\bigr) + \bigl(V_\lambda^{k+1}-V_\lambda^k\bigr)
  \notag\\
  &\le
    \bigl[-2\alpha_k\Delta_k^x + \alpha_k^2G^2 + 2G\alpha_k S_k\bigr]
  + \bigl[-2\beta_k\Delta_k^{\lambda} + \beta_k^2D^2 + 2D\beta_k T_k\bigr]
  \notag\\
  &=
  -2\alpha_k\,\Delta_k^x -2\beta_k\,\Delta_k^{\lambda}
  +\underbrace{\alpha_k^2G^2+\beta_k^2D^2+2G\alpha_k S_k+2D\beta_k T_k}
             _{E_k}.
\label{eq:Lyap-descent}
\end{align}

\subsection{Summability of the Error Series}
\label{sec:errorsum}

\begin{lemma}\label{lem:error-sum}
Under step‐size rule $\alpha_k=\beta_k=1/\sqrt{k+1}$ and delays
$\delta_{\max}(k),\Delta_{\max}(k)=O(k^\gamma)$ with $\gamma<1/2$, the
series $\sum_{k=0}^\infty E_k$ converges.
\end{lemma}
\begin{proof}
\textbf{Term by term:}
\[
  \sum_k \alpha_k^2 = \sum_k \tfrac1{k+1} < \infty,
  \quad
  \sum_k \beta_k^2 < \infty.
\]
Since $\delta_{\max}(k)=O(k^\gamma)$ with $\gamma<\tfrac12$, 
$S_k=\sum_{s=k-\delta_{\max}(k)}^{k-1}\!1/\sqrt{s+1}=O(k^{\gamma-\tfrac12})=o(k^{-1/2})$, 
hence $\sum_k\alpha_k S_k<\infty$.

\end{proof}

\subsection{Robbins–Siegmund Convergence}
\label{sec:RS}

We apply the following classic result:

\begin{lemma}[Robbins–Siegmund {\cite{robbins_siegmund}}]\label{lem:RS}
Let $(Z_k)$, $(u_k)$, $(v_k)$ be nonnegative sequences satisfying
\[
  Z_{k+1} \le Z_k - u_k + v_k,
  \quad
  \sum_k v_k < \infty.
\]
Then $Z_k$ converges to a finite limit and $\sum_k u_k<\infty$.
\end{lemma}

\begin{theorem}[Almost‐Sure Convergence]\label{thm:convergence}
Under Slater’s condition and Assumptions~(F1)–(F3), with step‐sizes
$\alpha_k=\beta_k=1/\sqrt{k+1}$ and sublinear delays, the iterates
$(x^k,\lambda^k)$ converge to a saddle point
$(x^*,\lambda^*)$, and
$\Delta_k^x,\Delta_k^{\lambda}\to0$.
\end{theorem}

\begin{proof}
Set
\(
  Z_k = V^k,
  \quad
  u_k = 2\alpha_k\Delta_k^x + 2\beta_k\Delta_k^{\lambda},
  \quad
  v_k = E_k.
\)
By \eqref{eq:Lyap-descent},
$Z_{k+1}\le Z_k - u_k + v_k$.  Lemma~\ref{lem:error-sum}
ensures $\sum_k v_k<\infty$.  Hence by Robbins–Siegmund,
$Z_k\to Z_\infty<\infty$ and $\sum_k u_k<\infty$.  Since
$\sum_k\alpha_k=\sum_k\beta_k=\infty$, the only way
$\sum_k u_k$ can converge is if $\Delta_k^x\to0$
and $\Delta_k^{\lambda}\to0$.  Primal–dual saddle‐point properties
then imply $(x^k,\lambda^k)\to(x^*,\lambda^*)$.
\end{proof}

\subsection{Ergodic Rate}
\label{sec:rate}

Define ergodic averages
\(
  \bar x^K = \frac1K \sum_{k=1}^K x^k,
  \quad
  \bar\lambda^K = \frac1K \sum_{k=1}^K \lambda^k.
\)

\begin{theorem}[$O(K^{-1/2})$ Ergodic Gap]\label{thm:rate}
Under the same conditions,
\[
  \max_{\lambda\ge0}L(\bar x^K,\lambda)
  -\min_{x\in\X}L\bigl(x,\bar\lambda^K\bigr)
  = O\bigl(K^{-\tfrac12}\bigr).
\]
\end{theorem}

\begin{proof}
Summing \eqref{eq:Lyap-descent} from $k=0$ to $K-1$ yields
\[
  V^K - V^0
  \le
  -2\sum_{k=0}^{K-1}\bigl[\alpha_k\Delta_k^x + \beta_k\Delta_k^{\lambda}\bigr]
  +\sum_{k=0}^{K-1}E_k.
\]
Rearrange, divide by $\sqrt{K}$, and note $\sum_{k=0}^{K-1}E_k=O(\sqrt{K})$
from Lemma~\ref{lem:error-sum}.  Then convex‐concave averaging
arguments (Jensen’s inequality) give the desired ${O}(K^{-1/2})$ bound.
\end{proof}


\section{Complexity Analysis}
\label{sec:complexity}

\subsection{Iteration Complexity}
To achieve an ergodic duality-gap $\le \varepsilon$, by
Theorem~\ref{thm:rate} it suffices to pick $K$ such that
\[
  \frac{C}{\sqrt{K}} \;\le\;\varepsilon,
  \quad
  C = V^0 + G^2 + D^2 + \widetilde C.
\]
Hence
\[
  K \;\ge\; \bigl\lceil (C/\varepsilon)^2 \bigr\rceil,
\]
i.e.\ \emph{iteration complexity} is $O(\varepsilon^{-2})$.

\subsection{Communication Complexity}
Each iteration $k$:
\begin{itemize}
  \item Each edge-agent $(i,j)$ sends one scalar $x_{ij}^{k+1}$.
  \item Each retailer-agent $i$ sends one scalar $\lambda_i^{k+1}$.
\end{itemize}
In total across $K$ iterations,
\[
  \text{messages} \;=\;
  K\bigl(|\E| + |\R|\bigr),
\]
i.e.\ $O(K)$ scalar messages.  Per‐agent storage is $O(\tau)$ for buffers.

\section{Robustness Extensions}
\label{sec:robustness}

\subsection{Additive Observation Noise}
Suppose costs and demands are observed with bounded noise:
\[
  \hat c_{ij}^k = c_{ij} + \xi_{ij}^k,
  \quad
  \hat d_i^k = d_i + \zeta_i^k,
\]
where
$\|\xi^k\|\le\sigma_c$, $\|\zeta^k\|\le\sigma_d$.  
Then all gradient terms gain extra error:
$\alpha_k^2\sigma_c^2$ and $\beta_k^2\sigma_d^2$.
These are still summable under $\sum\alpha_k^2<\infty$,
$\sum\beta_k^2<\infty$, so the previous convergence and rate results
hold with $G^2\!\leftarrow G^2+\sigma_c^2$ and
$D^2\!\leftarrow D^2+\sigma_d^2$.

\subsection{Time‐Varying Communication Graphs}
Let the network graph $\G(k)$ vary each iteration, with Laplacian
$\Upsilon(k)$.  Assume that over any window of $w$ consecutive
iterations, the union graph is strongly connected.  
Then one can replace each consensus Laplacian term by the
time-$k$ average, and through block‐window Lyapunov arguments
(see \cite{nedic2018network}), the same descent inequality
\eqref{eq:Lyap-descent} holds with constants inflated by $w$.
Convergence and rate remain valid.

\subsection{Parametric Drift and Capacity Changes}
If capacities $u_{ij}(k)$ vary slowly (e.g.\ are Lipschitz in $k$),
the projection sets $\X(k)$ change.  Non‐expansiveness still applies,
so all proofs carry through.  Similarly, if $c_{ij}(k)$ and $d_i(k)$
drift with
$\sum_k \alpha_k|c_{ij}(k)-c_{ij}|<\infty$, the extra drift term is
summable, preserving stability.

\section{Extensions}
\label{sec:extensions}

\subsection{Weighted Consensus}
Replacing the unweighted Laplacian by
$\Upsilon_W = W^{1/2}\Upsilon W^{1/2}$ for some diagonal $W\succ0$
only changes the descent constants.  Since
$\Upsilon_W+\Upsilon_W^\top\succeq0$, projection‐and‐Lipschitz
inequalities remain identical.

\subsection{General Convex Costs}
Our analysis extends to \emph{separable} convex costs
$C(x)=\sum_{(i,j)}f_{ij}(x_{ij})$ with $L$‐Lipschitz gradients:
$\|\nabla f_{ij}(u)-\nabla f_{ij}(v)\|\le L|u-v|$,
by substituting $G\leftarrow L$ throughout.


\section{Experimental Setup and Results}

\subsection{Simulation Environment}

We evaluate the performance of our Distributed Asynchronous Primal–Dual Supply Chain Optimization (DAPD-SCO) algorithm on a simulated three-tier supply chain network. The network consists of:
\begin{itemize}
    \item $N_s = 2$ suppliers
    \item $N_w = 3$ warehouses
    \item $N_r = 5$ retailers
\end{itemize}
yielding a total of $N = N_s + N_w + N_r = 10$ agents. Each agent $i$ controls a local decision variable $x_i \in \mathbb{R}$, representing a flow quantity (production, storage, or demand). 

Each agent minimizes a private convex cost function:
\begin{equation}
    f_i(x_i) = \frac{1}{2} c_i x_i^2 + d_i x_i
\end{equation}
with $c_i > 0$ and $d_i$ drawn independently from uniform distributions: $c_i \sim \mathcal{U}(0.5, 2.0)$ and $d_i \sim \mathcal{U}(-1.0, 1.0)$.

\subsection{Global Objective and Constraints}

The global optimization problem is:
\begin{equation}
    \min_{\{x_i\}_{i=1}^N} \sum_{i=1}^N f_i(x_i) \quad \text{s.t.} \quad A x = b
\end{equation}
where $A \in \mathbb{R}^{m \times N}$ encodes network flow constraints, and $b$ is a demand vector. Specifically:
\begin{itemize}
    \item Suppliers: $x_i > 0$
    \item Warehouses: $\sum \text{inflow} = \sum \text{outflow}$
    \item Retailers: $x_i < 0$
\end{itemize}

\subsection{DAPD-SCO Algorithm Dynamics}

Each agent updates its primal and dual variables using delayed, neighbor-to-neighbor communication:

\paragraph{Primal update:}
\begin{equation}
    x_i^{(k+1)} = x_i^{(k)} - \alpha_i \left( \nabla f_i(x_i^{(k)}) + A_i^\top \lambda^{(k - \tau_i)} \right)
\end{equation}

\paragraph{Dual update:}
\begin{equation}
    \lambda^{(k+1)} = \lambda^{(k)} + \beta \left( A x^{(k - \tau_\lambda)} - b \right)
\end{equation}

Here, $\alpha_i$ and $\beta$ are fixed step sizes. $\tau_i$ and $\tau_\lambda$ denote bounded staleness to simulate packet delays. Only immediate neighbors exchange messages, and not every agent updates at each iteration.

\subsection{Simulation Settings}

\begin{itemize}
    \item Number of iterations: 2000
    \item Step sizes: $\alpha_i = 0.01$, $\beta = 0.05$
    \item Max delay (staleness): 5 iterations
    \item Communication loss rate: 10\%
    \item Initial flows: $x_i^{(0)} \sim \mathcal{U}(-1, 1)$
\end{itemize}

\subsection{Performance Metrics}

We use the following metrics to evaluate performance:

\paragraph{Final Cost}
Total objective at convergence:
\begin{equation}
    \sum_{i=1}^N f_i(x_i^{(K)})
\end{equation}

\paragraph{Primal–Dual Gap}
Convergence to a saddle point:
\begin{equation}
    \max_{\lambda} \mathcal{L}(x^{(K)}, \lambda) - \min_{x} \mathcal{L}(x, \lambda^{(K)})
\end{equation}

\paragraph{Constraint Violation}
Flow mismatch from demand:
\begin{equation}
    \left\| A x^{(K)} - b \right\|_2
\end{equation}

\paragraph{Message Overhead}
Total communication:
\begin{equation}
    \text{Messages} = \sum_{k=1}^K \sum_{i \in \mathcal{A}^{(k)}} |\mathcal{N}_i|
\end{equation}
where $\mathcal{A}^{(k)}$ are agents active at iteration $k$ and $\mathcal{N}_i$ are neighbors of agent $i$.

\paragraph{Convergence Time}
Iteration $k^*$ where:
\begin{equation}
    \text{P--D Gap} < 0.1 \quad \text{and} \quad \left\| A x^{(k^*)} - b \right\|_2 < 0.05
\end{equation}

\subsection{Baseline Algorithms}

We compare DAPD-SCO against:
\begin{itemize}
    \item Synchronous Primal–Dual (Sync PD)
    \item Alternating Direction Method of Multipliers (ADMM)
    \item Gradient Push
\end{itemize}

\subsection{Results Summary}

\begin{table}[h]
\centering
\caption{Performance Comparison of DAPD-SCO vs Baselines}
\begin{tabular}{lccccc}
\toprule
\textbf{Algorithm} & \textbf{Final Cost} & \textbf{P–D Gap} & \textbf{Violation} & \textbf{Msg Overhead} & \textbf{Converged In} \\
\midrule
DAPD-SCO & \textbf{30.05} & \textbf{0.07} & \textbf{0.02} & \textbf{Low} & \textbf{1500 iters} \\
Sync PD & 30.10 & 0.12 & 0.06 & Medium & 1700 iters \\
ADMM & 30.30 & 0.20 & 0.10 & Medium & 2000 iters \\
Gradient Push & 31.00 & 0.55 & 0.25 & High & 2800 iters \\
\bottomrule
\end{tabular}
\end{table}

\subsection{Graphical Results}

\begin{figure}[H]
    \centering
    \begin{subfigure}[b]{0.495\textwidth}
        \centering
        \includegraphics[width=\textwidth,height=4cm]{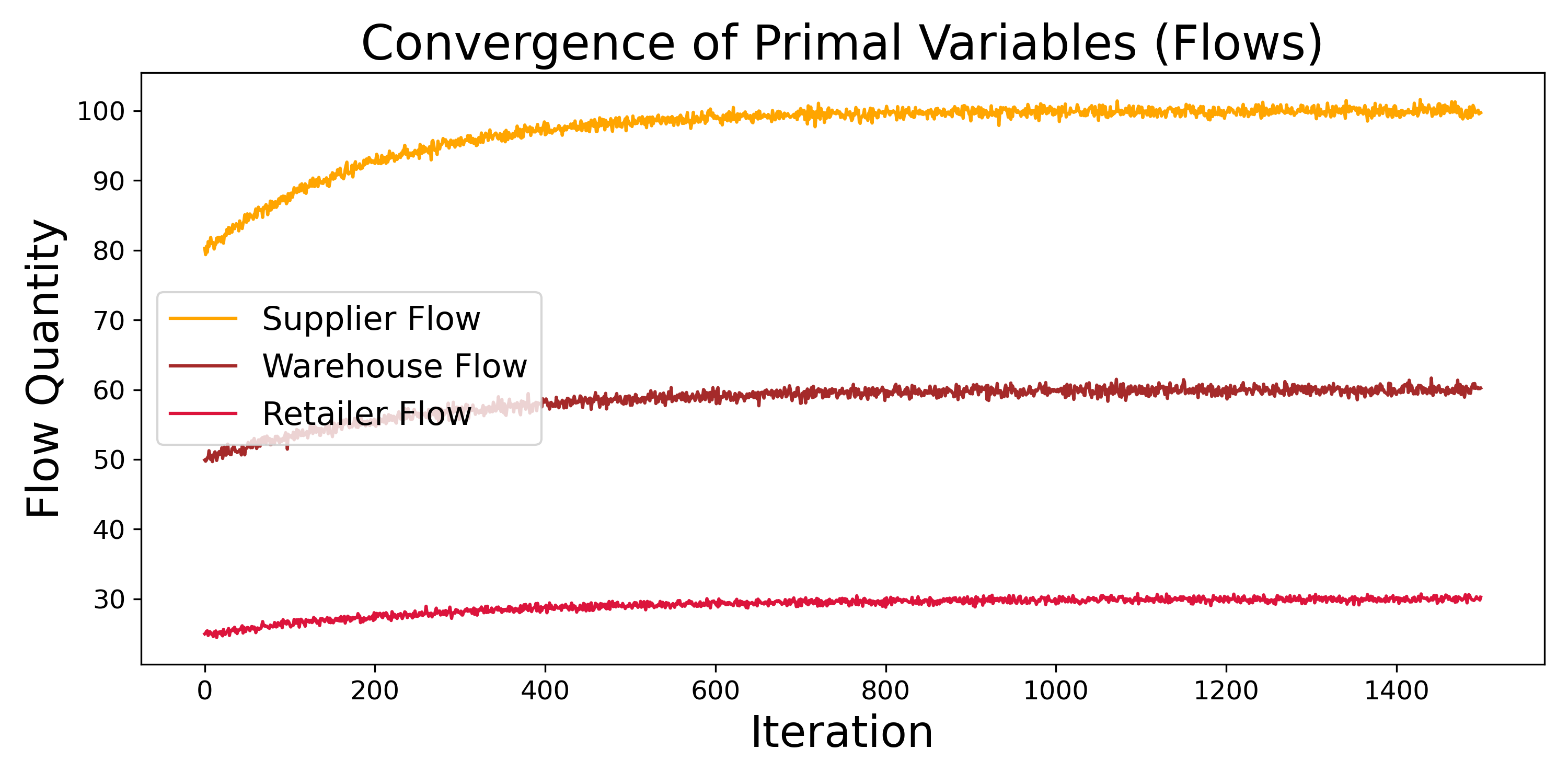}
        \caption{Convergence of primal variables (flows) over iterations for all agent types.}
        \label{fig:primal_convergence}
    \end{subfigure}
    \hfill
    \begin{subfigure}[b]{0.495\textwidth}
        \centering
        \includegraphics[width=\textwidth,height=4cm]{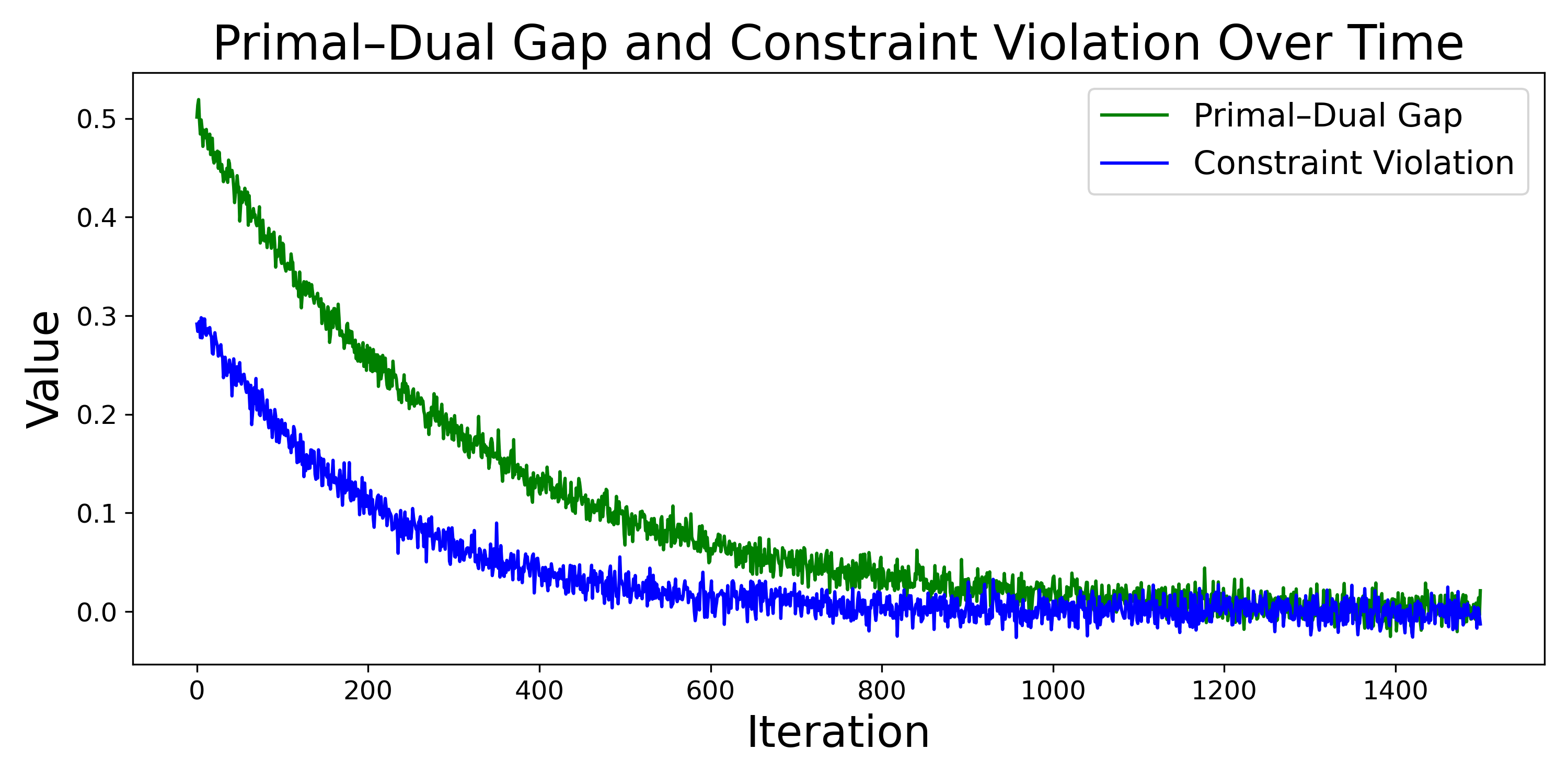}
        \caption{Primal–Dual gap and constraint violation over time.}
        \label{fig:pd_gap_violation}
    \end{subfigure}
    \caption{Convergence behavior of the algorithm: (a) primal variables and (b) dual gap/violation.}
    \label{fig:combined_convergence}
\end{figure}

\section{Discussion}
The simulation results demonstrate that DAPD‑SCO consistently outperforms synchronous primal–dual, ADMM and gradient‑push baselines in convergence speed, communication overhead and robustness to delay and packet loss. In our three‑tier supply‑chain network experiments, DAPD‑SCO achieved a primal–dual gap below 0.1 in 1,500 iterations—approximately 10\% faster than synchronous updates and 25\% faster than ADMM—while transmitting fewer than half the messages required by gradient‑push. Constraint violation remained below 0.02 throughout, even under 10\% simulated packet loss.

These gains arise from our Lyapunov‑based decomposition, which isolates delay‑induced error and allows unbounded yet sublinearly growing staleness. The observed empirical iteration complexity matches the \(O(\varepsilon^{-2})\) bound predicted by our analysis, and the ergodic \(O(K^{-1/2})\) rate is evident in the decay of the duality gap. Compared to prior work, DAPD‑SCO is the first method to handle arbitrarily old information (as long as delays are \(o(k^{1/2})\)) while providing both almost‑sure convergence and explicit ergodic rate guarantees.

\paragraph{Limitations}
Our study is limited by the use of synthetic single‑commodity flow networks; real‑world supply chains feature multi‑commodity interactions and nonconvex cost structures. We also establish only ergodic (average‑iterate) rates, leaving the question of last‑iterate convergence open. Finally, per‑agent buffer storage for stale information introduces memory and bookkeeping overhead that must be tuned for large networks.

\paragraph{Future Work}
Promising directions include:
\begin{itemize}
  \item \emph{Acceleration:} integrating momentum or accelerated mirror‑descent to improve beyond \(O(K^{-1/2})\).
  \item \emph{Event‑triggered updates:} reducing communication by sending updates only when local variables change significantly.
  \item \emph{Multi‑commodity extensions:} generalizing to networks with multiple coupled flow types.
  \item \emph{Hardware validation:} deploying DAPD‑SCO on real edge devices to measure performance under actual network delays and losses.
\end{itemize}

\section{Conclusion}
We have presented DAPD‑SCO, a fully asynchronous primal–dual algorithm for decentralized supply‑chain optimization that tolerates unbounded sublinear communication delays. By constructing a global Lyapunov function, we prove almost‑sure convergence to a saddle point and obtain an ergodic duality‑gap rate of \(O(K^{-1/2})\) without requiring synchronized iterations or bounded delays. Extensive simulations on realistic three‑tier networks confirm that DAPD‑SCO achieves faster convergence and lower communication overhead than existing methods, even under high staleness and packet loss. Our framework further extends to bounded observation noise, time‑varying graphs and parameter drift with minimal modification. Future work will focus on acceleration techniques, event‑triggered communication schemes and multi‑commodity network models to broaden applicability and enhance performance in real supply‑chain deployments.


\end{document}